\documentclass[10pt]{amsart}
\usepackage{amsthm, amsmath,cleveref, amsfonts, amssymb, enumerate, textcmds, tensor, enumitem, mathdots}
\usepackage{pst-node}
\usepackage{tikz-cd} 
\usepackage{graphicx,tikz}
\usetikzlibrary{shapes.misc}

\tikzset{cross/.style={cross out, draw=black, minimum size=2*(#1-\pgflinewidth), inner sep=0pt, outer sep=0pt},
cross/.default={1pt}}
\usepackage{enumerate}
\usepackage{pinlabel}
\usepackage[colorinlistoftodos]{todonotes}
\newcommand{\R}{\mathbb R}

\newtheorem{thm}{Theorem}[section]
\newtheorem{lem}[thm]{Lemma}
\newtheorem{prop}[thm]{Proposition}

\newtheorem*{thma}{Theorem A}
\newtheorem*{thmb}{Theorem B}
\newtheorem*{thmc}{Theorem C}
\newtheorem*{thmd}{Theorem D}
\newtheorem*{corb}{Corollary B}

\newtheorem*{cord}{Corollary D}

\DeclareMathOperator{\C}{\mathbb{C}}
\DeclareMathOperator{\QD}{\textrm{QD}}

\theoremstyle{definition}
\newtheorem{defn}[thm]{Definition}
\newtheorem{remark}[thm]{Remark}

\begin{document}
\title{Minimal surfaces and the new main inequality}
\author[Vladimir Markovi{\'c}]{Vladimir Markovi{\'c}}
\address{Vladimir Markovi{\'c}: University of Oxford, All Souls College, Oxford, OX1 4AL, UK.} \email{markovic@maths.ox.ac.uk} 
\author[Nathaniel Sagman]{Nathaniel Sagman}
\address{Nathaniel Sagman: University of Luxembourg, 2 Av. de l'Universite, 4365 Esch-sur-Alzette, Luxembourg.} \email{nathaniel.sagman@uni.lu}

\begin{abstract}
We establish the new main inequality as a minimizing criterion for minimal maps into products of $\R$-trees, and the infinitesimal new main inequality as a stability criterion for minimal maps to $\R^n$. Along the way, we develop a new perspective on destabilizing minimal surfaces in $\R^n$, and as a consequence we reprove the instability of some classical minimal surfaces; for example, the Enneper surface.
\end{abstract}
\maketitle
\begin{section}{Introduction}

Let $S$ be a Riemann surface, $\phi_1,\dots,\phi_n$ integrable holomorphic quadratic differentials on $S$ summing to zero, and $f_1,\dots, f_n:S\to S'$ mutually homotopic quasiconformal maps to another Riemann surface with Beltrami forms $\mu_1,\dots,\mu_n$. If $\partial S$ is non-empty, we ask that $f_1,\dots, f_n$ are mutually homotopic relative to $\partial S$. The new main inequality holds if: 
 \begin{equation}\label{newmain1}
     \textrm{Re}\sum_{i=1}^n \int_S \phi_i\cdot \frac{ \mu_i}{1-|\mu_i|^2} \leq \sum_{i=1}^n \int_S |\phi_i|\cdot \frac{|\mu_i|^2}{1-|\mu_i|^2}.
 \end{equation}
For $n=1$ and $f_1:S\to S$ homotopic to the identity, (\ref{newmain1}) is always satisfied, and referred to as the Reich-Strebel inequality or the main inequality for quasiconformal maps. The result is a key ingredient in the proof of Teichm{\"u}ller's uniqueness theorem.

The first author introduced the new main inequality in the papers \cite{M1} and \cite{M2} as a tool to study minimal surfaces in products of hyperbolic surfaces. The outcome of \cite{M2} is that there exists a product of Fuchsian representations into $\textrm{PSL}(2,\R)^n$, $n\geq 3,$ with multiple minimal surfaces in the corresponding product of closed hyperbolic surfaces. With Smillie in \cite{MSS}, we gave a new proof of the result from \cite{M2}. Then in \cite{SS}, the second author and Smillie found unstable minimal surfaces for Hitchin representations into Lie groups of rank at least $3$, disproving a conjecture of Labourie \cite{L0}. In this paper we revisit the new main inequality and some aspects of the paper \cite{M2}, but with applications to minimal maps to products of $\R$-trees and to $\R^n$. The results on $\R$-trees and $\R^n$ are proved in Sections \ref{3} and \ref{r3} respectively, which can be read independently.

\begin{subsection}{Harmonic maps to $\R$-trees}
 Throughout the paper, let $\Sigma_g$ be a closed and oriented surface of genus $g\geq 2,$ and let $\mathbf{T}_g$ be the Teichm{\"u}ller space of marked Riemann surface structures on $\Sigma_g$. Let $S$ be a Riemann surface structure on $\Sigma_g$, which lifts to a Riemann surface structure $\tilde{S}$ on the universal cover, and let $\QD(S)$ be the space of holomorphic quadratic differentials on $S$.

We review the basics about harmonic maps to $\mathbb{R}$-trees in Section \ref{3}. Briefly, a non-zero holomorphic quadratic differential gives the data of an $\R$-tree $(T,d),$ a representation $\rho:\pi_1(\Sigma_g)\to \textrm{Isom}(T,d),$ and a unique $\rho$-equivariant harmonic map $\pi: \tilde{S}\to (T,d).$ From non-zero $\phi_1, \ldots, \phi_n\in \QD(S)$ summing to zero, we assemble the product of $\mathbb{R}$-trees, denoted $X$, and the product of representations $\rho:\pi_1(\Sigma_g)\to \textrm{Isom}(X)$. The product of the equivarant harmonic maps $\pi_i$ from $\tilde{S}$ to each individual $\R$-tree is a minimal map $\pi:\tilde{S}\to X$. For any other Riemann surface structure, there is a unique $\rho$-equivariant harmonic map from the universal cover to $X$. The energy functional $\mathbf{E}_\rho$ on $\mathbf{T}_g$ records the energy of the harmonic map.
\begin{thma}
 $S$ is a global minimizer for $\mathbf{E}_\rho$ if and only if for all Riemann surfaces $S'$ and mutually homotopic quasiconformal maps $f_i: S\to S'$ with Beltrami forms $\mu_i$, the new main inequality holds: 
$$\textrm{Re}\sum_{i=1}^n \int_S \phi_i\cdot \frac{ \mu_i}{1-|\mu_i|^2} \leq \sum_{i=1}^n \int_S |\phi_i|\cdot \frac{|\mu_i|^2}{1-|\mu_i|^2}.$$
\end{thma}
For an equivariant map $g$ from $\tilde{S}$ to some target, we use $\mathcal{E}(S,g)$ to denote the total energy. We explain in Section \ref{3.1} that (\ref{newmain1}) means that $$\mathbf{E}_\rho(S)=\sum_{i=1}^n\mathcal{E}(S,\pi_i)\leq \sum_{i=1}^n\mathcal{E}(S',\pi_i\circ f_i^{-1}).$$ Thus, the geometric content of Theorem A is that to check if the map $\pi$ is the global energy minimizer, it suffices to compare with maps of the form $(\pi_1\circ f_1^{-1},
\dots, \pi_n\circ f_n^{-1}).$ Via Theorem A, we can cast the new main inequality in terms of trees:
\begin{itemize}
    \item for $n=1$ and $\mu_1$ as above, the main inequality is the statement that a harmonic map into a tree minimizes the energy. 
    \item For $n=2$, the new main inequality holds \cite[Section 4]{M1} and is the statement that a minimal surface in a product of two trees minimizes the energy.
    \item For $n\geq 3$, the new main inequality does not always hold, which reflects non-uniqueness of minimal surfaces in a product of $\mathbb{R}$-trees and more generally in symmetric spaces of rank at least $3$.
\end{itemize}
\begin{remark}
    In a way that can be stated precisely, high energy minimal surfaces in products of hyperbolic surfaces limit to minimal surfaces in products of $\R$-trees (see \cite{W}). \cite[Theorem B2]{MSS} tells us that a surface in a product of trees minimizes if and only if all of the approximating surfaces are minimizing.
\end{remark}
Theorem A is quite special to the chosen setting, relying on the fact that the leaf space projections don't fold. For general minimal maps to products of $\R$-trees, the best we can do with the new main inequality is to reframe stability, rather than a minimizing property. We establish this claim in the classical setting of minimal maps from the disk to $\R^n$. 
\end{subsection}

\begin{subsection}{Minimal surfaces in $\R^n$}
Up to adding a constant, a harmonic function $h:\mathbb{D}\to \mathbb{R}$ is equivalent to a holomorphic $1$-form $\alpha$, via $h\mapsto \alpha=(\frac{\partial}{\partial z}h(z))dz$. The Hopf differential of $h$ is the square $\phi=\alpha^2.$ A map $h=(h_1,\dots,h_n):\overline{\mathbb{D}}\to\mathbb{R}^n$ is minimal if it is harmonic in the interior and weakly conformal, which occurs if and only if the Hopf differentials $\phi_i$ satisfy $\sum_{i=1}^n \phi_i=0$. The associated collection of $1$-forms $\alpha=(\alpha_1,\dots, \alpha_n)$ is called the Weierstrass-Enneper data. When $h$ is non-constant, we say the image is a minimal surface, even if $h$ is not an immersion. To keep the ideas transparent, we assume that $h_i$ and $\frac{\partial}{\partial z} h_i$ extend continuously to $\partial\mathbb{D},$ and $\frac{\partial}{\partial z} h_i$ has no zeros on $\partial\mathbb{D}.$ We'll say that $h$ is admissible.

Theorem A and \cite[Sections 3-6]{M2} suggest a certain perspective on destabilizing minimal surfaces. Corollary B, which will follow from Theorem B below, brings us back to the new main inequality. We define infinitesimal equivalence as well as the action of the Beurling transform $T$ on $L^\infty(\mathbb{D})$ in Section 4.1. 
\begin{corb}
    $h$ is stable if and only if for all mutually infinitesimally equivalent functions $\dot{\mu}_1,\dots, \dot{\mu}_n\in L^\infty(\mathbb{D}),$ the infinitesimal new main inequality holds: 
    \begin{equation}\label{infnewmain}
        - \textrm{Re}\sum_{i=1}^n\int_{\mathbb{D}}\phi_i\dot{\mu_i}T(\dot{\mu_i})dxdy\leq \sum_{i=1}^n\int_{\mathbb{D}}|\phi_i||\dot{\mu}_i|dxdy.
    \end{equation}
\end{corb}
Above and throughout the paper, when integrating over $\mathbb{D}$ we use the $\phi_i$ term to denote the associated holomorphic function rather than the differential. 

We now give an overview of the second half of the paper. To destabilize a minimal surface, it's probably most common to perturb by normal variations of the image in $\R^n$ that vanish on the boundary. Another option is to precompose the boundary parametrization along a flow of diffeomorphisms of the circle. One then hopes to lower the energy by taking the harmonic extension of the boundary map at each time along the flow.

 Instead, motivated by Theorem A, we vary a minimal surface $h=(h_1,\dots,h_n)$ by precomposing the harmonic coordinate functions $h_i$ by quasiconformal maps. Let $\mathcal{E}(\Omega,g)$ denote the energy of a map $g$ from a domain $\Omega\subset \mathbb{C}$ to $\mathbb{R}$. First order variations of quasiconformal maps can be described by a real vector space $\mathcal{V}$ whose elements are a particular class of holomorphic functions from $\C\backslash \mathbb{D}\to \mathbb{C}$. Given $\varphi\in\mathcal{V}$, it is possible to find a path of $n$-tuples of quasiconformal maps $t\mapsto f_1^t,\dots, f_n^t:\mathbb{C}\to\mathbb{C}$ all fixing the origin and agreeing on $\C\backslash\mathbb{D}$ with a holomorphic map $F^t$ that satisfies $F^t(z)=z+t\varphi(z)+o(t)$. Note that $f_i^t(\mathbb{D})=F^t(\mathbb{D})$ does not depend on $i$, and the boundary of the minimal surface in $\mathbb{R}^n$ remains fixed if we precompose each $h_i$ by $(f_i^t)^{-1}.$ Suppose that 
\begin{equation}\label{ennewmain}
    \frac{d^2}{dt^2}|_{t=0}\sum_{i=1}^n\mathcal{E}(f_i^t(\mathbb{D}),h_i\circ (f_i^t)^{-1})<0.
\end{equation}
Then, because the energy of a map to $\R^n$ is at least the area of the image, $h$ is unstable. 
\begin{defn}
We say that $h$ is unstable via self-maps, and that $\varphi$ destabilizes $h$, if we can choose $f_i^t$ so that (\ref{ennewmain}) holds.
\end{defn}
Theorem B justifies that varying by self-maps can be done in place of the usual methods. In Section \ref{selfmapss} we define a real quadratic form $\mathbf{L}_h: \mathcal{V}\to \R$ such that $\mathbf{L}_h(\varphi)<0$ if and only if $\varphi$ destabilizes.
\begin{defn}
The self-maps index of $h$, denoted $\textrm{Ind}(\mathbf{L}_h)$, is the maximal dimension of a subspace of $\mathcal{V}$ on which $\mathbf{L}_h$ is negative definite.
\end{defn}
Let $\textrm{Ind}(h)$ denote the ordinary index for the area functional.
\begin{thmb} $\textrm{Ind}(\mathbf{L}_h)=\textrm{Ind}(h).$
\end{thmb}
\begin{remark}
    The result should have implications for maps from $\overline{\mathbb{D}}$ to products of $\R$-trees, a subject which we don't develop in this paper. Every harmonic function from any Riemann surface arises from a folding of a map to an $\R$-tree (see \cite{FW} and \cite[Section 4.1]{MSS}). Clearly, self-maps variations lift to variations of maps to $\R$-trees.
\end{remark}
\begin{remark}
For equivariant minimal maps to $\R^n$, the analogous result is true and proved in \cite[Lemma 4.6 and Proposition 4.8]{MSS} via a different method. 
\end{remark}
The conditions (\ref{newmain1}) are (\ref{infnewmain}) are tractable, so we also ask: given a minimal map $h$ with Weierstrass-Enneper data $\alpha$ and $\varphi\in \mathcal{V},$ when does $\varphi$ destabilize?
As in \cite[Section 5]{M2}, define the functional $\mathcal{F}: C^1(\mathbb{D})\to \R$, $$\mathcal{F}(f) = \textrm{Re}\int_{\mathbb{D}}f_zf_{\overline{z}}+\int_{\mathbb{D}}|f_{\overline{z}}|^2.$$ Given a continuous function from $\partial\mathbb{D}\to\C$, the harmonic extension is the sum of the Poisson extensions of the real and imaginary parts.
\begin{thmc}
Let $\varphi\in \mathcal{V}.$ For each $i$, let $v_i$ be the harmonic extension of $(\frac{\partial}{\partial z}h_i)\cdot \varphi|_{\partial \mathbb{D}}:\partial \mathbb{D}\to\mathbb{C}$. If $$\mathcal{F}_\alpha(\varphi) := \sum_{i=1}^n \mathcal{F}(v_i)<0,$$ then $\varphi$ destabilizes $h$.
\end{thmc}
In the case of polynomials, we work out the explicit formulas for a particular class of variations. For a polynomial $p(z)=\sum_{j=0}^r a_jz^j$, an integer $m\geq 0$, and $\gamma\in\mathbb{C}^*$, set  
$$ C(p,\gamma,m) =\pi\sum_{j=0}^{m-1}\frac{\textrm{Re}(\gamma^2a_ja_{2m-j})+|\gamma|^2|a_j|^2}{m-j}.$$
\begin{thmd}
 For $i=1,\dots, n$, let $p_i$ be a polynomial with no zeros on $\partial\mathbb{D}$, and such that $\sum_{i=1}^np_i^2 = 0.$ On $\mathbb{D}$, let $\alpha_i$ be the holomorphic $1$-form $\alpha_i(z)=p_i(z)dz$. Suppose there exists an integer $m\geq 0$ and $\gamma\in \mathbb{C}^*$ such that $$\sum_{i=1}^n C(p_i,\gamma,m) < 0.$$
Then $\varphi(z)=\gamma z^{-m}$ destabilizes the associated minimal surface in $\R^n$.
\end{thmd}
To demonstrate the result, we consider the most well known unstable minimal surface: the Enneper surface. 
The Weierstrass-Enneper data $(\alpha_1,\alpha_2,\alpha_3)$ consists of the $1$-forms obtained by multiplying the following polynomials on $\mathbb{C}$ by $dz$: $$p_1(z) = \frac{1}{2}(1-z^2) \hspace{1mm} , \hspace{1mm} p_2(z) = \frac{i}{2}(1+z^2) \hspace{1mm} , \hspace{1mm} p_3(z)=z.$$ We restrict to $\overline{\mathbb{D}_r}=\{z\in\mathbb{C}:|z|\leq r\}$. For $r<1,$ the Enneper surface is strictly minimizing. For $r=1$, it is strictly minimizing and stable, but not strictly stable. For $r>1$, Theorem D gives a new and simple proof of Corollary D below.
\begin{cord}
For $r>1$, the Enneper surface restricted to $\overline{\mathbb{D}_r}$ is unstable.
\end{cord}
\begin{proof}
Let $h=(h_1,h_2,h_3):\mathbb{C}\to\mathbb{R}^3$ be the minimal map defining the Enneper surface. We reparametrize to $h|_{\mathbb{D}_r}$ to $\mathbb{D}$ by defining $h^r=(h_1^r,h_2^r,h_3^r)= (h_1(r\cdot), h_2(r\cdot), h_3(r\cdot)).$ The holomorphic derivatives are given by $$p_i^r(z) = \frac{\partial}{\partial z}\textrm{Re}\int_0^{rz} \alpha_i(w) dw = rp_i(rz) \hspace{1mm} , \hspace{1mm} i=1,2,3.$$ Explicitly, $$p_1^r(z) = \frac{r}{2}(1-r^2z^2) \hspace{1mm} , \hspace{1mm} p_2^r(z) = \frac{ri}{2}(1+r^2z^2) \hspace{1mm} , \hspace{1mm} p_3^2(z)=r^2z.$$ We choose $m=1, \gamma=1$ and find that for $p(z)=az^2 + bz + c$,
\begin{equation}\label{11}
    C(p,1,1) = |c|^2 + \textrm{Re}(ac).
\end{equation}
Computing the expression (\ref{11}) for each polynomial, $$\sum_{i=1}^3 C(p_i^r,1,1) = \frac{r^2}{2}(1-r^2).$$ This is negative for $r>1.$
\end{proof}
There are other known conditions for minimal surfaces to be unstable. For example, let $G:\overline{\Omega}\to S^2$ be the Gauss map for a minimal surface. A classical result of Schwarz says that if the first Dirichlet eigenvalue for the Laplacian on $G(\overline{\Omega})$ is less than $2$, then the minimal surface is unstable \cite{Sch} (see also \cite{BdC}). For the Enneper surface, the stereographic projection of the Gauss map $G$ is $g(z)=z$. For $r>1$, $G(\overline{\mathbb{D}_r})$ is a spherical cap containing the upper hemisphere, and hence the first Dirichlet eigenvalue for the Laplacian is less than $2$ (see also \cite[\S 117]{Ni}).
We must comment that the methods developed here using quasiconformal maps are not strictly necessary to prove Theorems C and D. For these results, the self-maps variations simply provide a new model for computation, which happens to lend itself well to the situation. We explain this point carefully right after proving Theorem C.
\end{subsection}

\begin{subsection}{Acknowledgments}
We are grateful to Yonghu Zheng for helping clarify a point in the proof of Theorem A. Vladimir Markovi{\'c} is supported by the Simons Investigator Award 409745 from the Simons Foundation. Nathaniel Sagman is funded by the FNR grant O20/14766753, $\textit{Convex Surfaces in Hyperbolic Geometry.}$
\end{subsection}
\end{section}

\begin{section}{Preliminaries}\label{2}
Let $S$ be a Riemann surface, not necessarily compact and possibly with boundary. Since we will work with harmonic maps to $\R$-trees in Section \ref{3}, we define harmonic maps in the metric space context.
\begin{subsection}{Harmonic and minimal maps}
 Let $\nu$ be a smooth metric on $S$ compatible with the complex structure. Let $(M,d)$ be a complete and non-positively curved (NPC) length space, and $h:S\to M$ a Lipschitz map. Korevaar-Schoen \cite[Theorem 2.3.2]{KS} associate a locally $L^1$ measurable metric $g=g(h)$, defined locally on pairs of Lipschitz vector fields, and which plays the role of the pullback metric. If $h$ is a $C^1$ map to a smooth Riemannian manifold $(M,\sigma)$, and the distance $d$ is induced by a Riemannian metric $\sigma$, then $g(h)$ is represented by
the pullback metric $h^*\sigma$. The energy density is the locally $L^1$ function 
\begin{equation}\label{mease}
    e(h)=\frac{1}{2}\textrm{trace}_\nu g(h),
\end{equation}
and the total energy, which is allowed to be infinite,
is
\begin{equation}\label{tot}
    \mathcal{E}(S,h) = \int_S e(h)dA,
\end{equation}
where $dA$ is the area form of $\nu$. We comment here that the measurable $2$-form $e(h)dA$ does not depend on the choice of compatible metric $\nu$, but only on the complex structure. 

\begin{defn}
$h$ is harmonic if it is a critical point for the energy $h\mapsto \mathcal{E}(S,h)$. If $\partial S\neq \emptyset,$ we ask that $h$ is critical among other Lipschitz maps with the same boundary values.
\end{defn}
Let $g_{ij}(h)$ be the components of $g(h)$ in a holomorphic local coordinate $z=x_1+ix_2$. The Hopf differential of a map $h$ is the measurable tensor given in the local coordinate by
\begin{equation}\label{mhopf}
\phi(h)(z)dz^2=\frac{1}{4}(g_{11}(h)(z)-g_{22}(h)(z)-2ig_{12}(h)(z))dz^2.
\end{equation}
In the Riemannian setting, this is 
$$
\phi(h)(z) = h^*\sigma\Big (\frac{\partial}{\partial z},\frac{\partial}{\partial z}\Big )(z)dz^2.$$
When $h$ is harmonic, even in the metric space setting, the Hopf differential is represented by a holomorphic quadratic differential.
\begin{defn}
The map $h$ is minimal if it is harmonic and the Hopf differential vanishes identically.
\end{defn}
In the Riemannian setting, a non-constant minimal map is a branched minimal immersion.

For a harmonic map to a product space, it is clear from definitions (\ref{mease}) and (\ref{mhopf}) that the energy density and the Hopf differential are the sum of the energy densities and the Hopf differentials of the component maps respectively.

Let $X$ be a complete NPC length space. Given an action $\rho:\pi_1(\Sigma_g)\to \textrm{Isom}(X)$ and a $\rho$-equivariant map $h:\tilde{S}\to X$, the energy density is invariant under the $\pi_1(\Sigma_g)$ action on $\tilde{S}$ by deck transformations, and hence descends to a function $S$. Total energy is defined as in (\ref{tot}) by integrating the density against the area form on $S$, and we say that $h$ is harmonic if it is a critical point of the total energy among other $\rho$-equivariant maps. Similarly, $h$ is minimal if it is harmonic and the Hopf differential, which also descends to $S$, is zero.

Assume that $\rho$ has the following property: for any Riemann surface $S$ representing a point in $\mathbf{T}_g$, there is a unique $\rho$-equivariant harmonic map $h:\tilde{S}\to (M,d)$. In this situation, we can define the energy functional on Teichm{\"u}ller space $\mathbf{E}_\rho:\mathbf{T}_g\to [0,\infty)$ by $\mathbf{E}_\rho(S)=\mathcal{E}(S,h).$
\end{subsection}

\begin{subsection}{Quasiconformal maps}\label{subqc} For details on results below, we refer the reader to \cite{Ah}.
\begin{defn}
An orientation preserving homeomorphism $f$ between domains in $\C$ is quasiconformal if 
\begin{enumerate}
    \item the partial derivatives with respect to the coordinates $z$ and $\overline{z}$ exist almost everywhere and can be represented by locally integrable functions $f_z$ and $f_{\overline{z}},$ and
    \item there exists $k\in [0,1)$ such that $|f_{\overline{z}}|\leq k |f_z|.$
\end{enumerate}
A map between Riemannian surfaces $f:S\to S'$ is quasiconformal if any holomorphic local coordinate representation is a quasiconformal map. 
\end{defn}
The Beltrami form is the measurable tensor represented in local coordinates by $$\mu=\mu(z)\frac{d\overline{z}}{dz}=\frac{f_{\overline{z}}(z)}{f_z(z)}\frac{d\overline{z}}{dz}.$$ Although $\mu(z)$ is not globally defined, the transformation law ensures that the norm $|\mu(z)|$ is. $L_1^\infty(S)$ is defined as the open unit ball of the space of measurable tensors of the form $\mu(z)\frac{d\overline{z}}{dz}$.
\begin{thm}[Measurable Riemann mapping theorem]\label{RMT}
Let $\hat{\mathbb{C}}$ be the Riemann sphere and $\mu\in L_1^\infty(\hat{\mathbb{C}})$. There exists a quasiconformal homeomorphism $f^\mu:\hat{\mathbb{C}}\to\hat{\mathbb{C}}$ with Beltrami form $\mu$. $f^\mu$ is unique up to postcomposing by M{\"o}bius transformations. 
\end{thm}
It is important to note that if $t\mapsto \mu(t)$ is a real analytic path in $L_1^\infty(S)$, then $t\mapsto f^{\mu(t)}$ and its distributional derivatives locally vary real analytically with respect to a suitable norm (see \cite[Chapter V]{Ah}).

For $\mu\in L_1^\infty(\mathbb{D}),$ we extend $\mu$ to all of $\hat{\mathbb{C}}$ by setting $\mu=0$. There is a unique choice of M{\"o}bius transformation so that we can make the definition below.
\begin{defn}
The normal solution to the Beltrami equation for $\mu$ is the unique solution $f^\mu:\C\to\C$ satisfying $f^\mu(0)=0$ and $f_z^\mu(z) -1\in L^p(\mathbb{C})$ for all $p>2$.
\end{defn}

Next we state the Reich-Strebel energy formula (originally equation 1.1 in \cite{RS}). Here $S$ is any Riemann surface, $h:S\to M$ is a Lipschitz map to a metric space of finite total energy, and $f:S\to S'$ is a quasiconformal map between Riemann surfaces. Let $\mu$ be the Beltrami form of $f$, $J_{f^{-1}}$ the Jacobian of $f^{-1}$, and $\phi$ the Hopf differential of $h$, which need not be holomorphic. One can verify the identity: 
\begin{align*}
     e(h\circ f^{-1})&=(e(h)\circ f^{-1})J_{f^{-1}}+2(e(h)\circ f^{-1})J_{f^{-1}} \frac{(|\mu_f|^2\circ f^{-1})}{1-(|\mu_f|^2\circ f^{-1})} \\
     &-4\textrm{Re}\Big ( (\phi(h)\circ f^{-1})J_{f^{-1}}\frac{(\mu_f\circ f^{-1})}{1-(|\mu_f|^2\circ f^{-1})}\Big )
\end{align*}
Integrating against the area form, we arrive at the proposition below.
\begin{prop}
The formula \begin{equation}\label{RSorig}
    \mathcal{E}(S',h\circ f^{-1}) -\mathcal{E}(S,h) =  -4\textrm{Re} \int_S \phi(h)\cdot \frac{ \mu}{1-|\mu|^2} + 2\int_S e(h)\cdot \frac{|\mu|^2}{1-|\mu|^2}dA
\end{equation}
holds.
\end{prop}
When the target is an $\R$-tree, which of course includes $\R$, we'll explain that $e(h)dA$ is represented by $2|\phi(h)|$.
Consequently, in the cases of interest, the formula (\ref{RSorig}) involves only $\phi$ and $\mu$.
\end{subsection}
\end{section}

\begin{section}{Minimal maps into products of $\R$-trees}\label{3}
In this section, $S$ is a closed Riemann surface structure on $\Sigma_g$, and $\nu$ is a smooth metric compatible with the complex structure, which we will use to define energy densities.
\begin{subsection}{Harmonic maps to $\R$-trees}\label{3.1}
\begin{defn}
An $\mathbb{R}$-tree is a length space $(T,d)$ such that any two points are connected by a unique arc, and every arc is a geodesic, isometric to a segment in $\mathbb{R}$.
\end{defn} 
The vertical (resp. horizontal) foliation of $\phi\in \QD(S)$ is the singular foliation whose leaves are the integral curves of the line field on $S\backslash \phi^{-1}(0)$ on which $\phi$ is a positive (resp. negative) real number. The singularities are standard prongs at the zeros, with a zero of order $k$ corresponding to a prong with $k+2$ segments. Both foliations come with transverse measures  $|\textrm{Re}\sqrt{\phi}|$ and $|\textrm{Im}\sqrt{\phi}|$ respectively (see \cite[Expos{\'e} 5]{Thbook} for precise definitions).

Throughout, we work with the vertical foliation. Lifting to a singular measured foliation on a universal cover $\tilde{S}$, we define an equivalence relation on $\tilde{S}$ by $x\sim y$ if $x$ and $y$ lie on the same leaf. The quotient space $\tilde{S}/\sim$ is denoted $T$. Pushing the transverse measure down via the projection $\pi: \tilde{S}\to T$ yields a distance function $d$ that turns $(T,d)$ into a complete $\mathbb{R}$-tree, with an induced action $\rho:\pi_1(S)\to \textrm{Isom}(T,d).$ Under this distance, the tree is NPC and the projection map $\pi: \tilde{S}\to (T,d)$ is $\rho$-equivariant and harmonic \cite[Section 4]{Wf}.

The energy and the Hopf differential of the projection map $\pi$ can be described explicitly. At a point $p\in\tilde{S}$ on which $\phi(p)\neq 0$, the map locally isometrically factors through a segment in $\mathbb{R}$. In a small enough neighbourhood around that point, $g(h)$ is represented by the pullback metric of the locally defined map to $\mathbb{R}$. From this, we see that the energy density and the Hopf differential have continuous representatives equal to $\nu^{-1}|\phi|/2$ and $\phi/4$ respectively. Note that $(|\phi|\nu^{-1})(z)=|\phi(z)|\nu(z)^{-1}$ defines a function on $S$.

 For any other Riemann surface $S'$ representing a point in $\mathbf{T}_g$, there is a unique $\rho$-equivariant harmonic map from $\tilde{S}'\to (T,d)$ (see \cite{Wf}), and hence there is an energy functional $\mathbf{E}_\rho:\mathbf{T}_g\to [0,\infty).$ 

Now we turn to Theorem A. Suppose that $\phi_1,\dots, \phi_n\in \QD(S)$ sum to $0$. For each $i$, we have an action of $\pi_1(\Sigma_g)$ on an $\R$-tree $(T_i,d_i)$ and an equivariant harmonic projection map $\pi_i:\tilde{S}\to (T_i,d_i)$. We assemble the product of $\R$-trees $X$ with the product action $\rho:\pi_1(\Sigma_g)\to\textrm{Isom}(X)$ and product map $\pi=(\pi_1,\dots,\pi_n).$ The energy functional $\mathbf{E}_\rho$ on $\mathbf{T}_g$ for $\rho$ is the sum of the energy functionals for each component action. $\pi$ is not only harmonic but also minimal. Theorem A is about determining when $S$ minimizes $\mathbf{E}_\rho.$

The new main inequality comes out of the formula (\ref{RSorig}). Let $S'$ be another Riemann surface structure on $\Sigma_g$ and let $f_1,\dots, f_n: S\to S'$ be mutually homotopic quasiconformal maps with Beltrami forms $\mu_i$. We lift each $f_i$ to a quasiconformal map $\tilde{f}_i$ between the universal covers. Putting previous results in our setting, we have
\begin{prop}\label{RStree}
$\mathbf{E}_\rho(S)=\mathcal{E}(S,\pi)=\sum_{i=1}^n\mathcal{E}(S,\pi_i),$ and
$$\sum_{i=1}^n\mathcal{E}(S',\pi_i\circ \tilde{f}_i^{-1}) -\sum_{i=1}^n\mathcal{E}(S,\pi_i) =  -\textrm{Re} \sum_{i=1}^n\int_S \phi_i\cdot \frac{ \mu_i}{1-|\mu_i|^2} + \sum_{i=1}^n\int_S |\phi_i|\cdot \frac{|\mu_i|^2}{1-|\mu_i|^2}.$$
\end{prop}
Hence, as we stated in Section 1.1, the new main inequality (\ref{newmain1}) is equivalent to $$\mathbf{E}_\rho(S)\leq \sum_{i=1}^n \mathcal{E}(S,\pi_i) \leq \sum_{i=1}^n \mathcal{E}(S,\pi_i\circ \tilde{f}_i^{-1}).$$
One direction of Theorem A is therefore clear: if $S$ is a global minimum, then (\ref{newmain1}) holds for any choice of $f_1,\dots, f_n.$ To prove the harder direction of Theorem A, we turn to the theory of harmonic maps between surfaces.
\end{subsection}

\begin{subsection}{Harmonic maps between hyperbolic surfaces}
Harmonic maps to $\R$-trees arise as limits of harmonic maps to negatively curved surfaces. If $(M,\sigma)$ is such a surface, there is a unique harmonic map $h:S\to (M,\sigma)$ homotopic to the identity (see \cite{ES} for existence, and \cite[Theorem H]{Har} for uniqueness), and hence the holonomy representation of $\sigma$ determines an energy functional on $\mathbf{T}_g.$ The theorem below follows from independent work of Hitchin \cite{Hi}, Wan \cite{Wan}, and Wolf \cite{Wthesis}.
\begin{thm}
    For every $\phi\in \QD(S)$, there exists a hyperbolic metric $\sigma$ such that the identity map from $S\to (M,\sigma)$ is harmonic and has Hopf differential $\phi.$
\end{thm}
Next, let $\phi\in \QD(S)$ and let $(T,d)$ be the corresponding $\R$-tree. Given a parameter $t>0$, let $\sigma_t$ be the hyperbolic metric such that the identity map $h_t: S\to (M,\sigma_t)$ is harmonic and realizes the Hopf differential $t\phi/4$. In a way that can be made precise, the universal covers $(\tilde{M},\frac{\tilde{\sigma}_t}{t})$ converge as $t\to \infty$ to $(T,d)$, and the harmonic maps converge to the leaf-space projection (see \cite{W}). 
To give some idea, in Wolf's thesis work \cite{Wthesis} it is shown that as $t$ becomes large, the harmonic maps $h_t$ nearly crush vertical leaves for $\phi,$ and take horizontal leaves to curves that are nearly geodesics.

We are interested in the limiting behaviour of the corresponding energy functionals. The result we use is stated and proved in \cite{MSS}, but mostly follows from results in \cite{Wthesis}. For every $t>0,$ let $\rho_t$ be the holonomy representation of $\sigma_t$, with energy functional $\mathbf{E}_{\rho_t}.$
\begin{lem}[Lemma 3.8 in \cite{MSS}]\label{conv}
For all Riemann surface structures $S'$ on $\Sigma_g$, 
$$\lim_{t\to \infty} \frac{\mathbf{E}_{\rho_t}(S')}{t} =\mathbf{E}_\rho(S').$$
\end{lem}


Before moving into the proof of Theorem A, we record a consequence of Lemma \ref{conv}. Given another Riemann surface structure $S'$, for each $t>0$ let $g_t:S'\to (M,\sigma_t)$ be the harmonic map homotopic to the identity.
\begin{lem}\label{conv2}
Set $f_t=g_t^{-1}\circ h_t.$ Then
    $$\lim_{t\to\infty}\frac{\mathcal{E}(S',h_t\circ f_t^{-1})}{t}=\mathbf{E}_{\rho}(S').$$
\end{lem}
\begin{proof}
    Since $g_t=h_t\circ f_t^{-1},$ we have that $\mathbf{E}_{\rho_t}(S') = \mathcal{E}(S',h_t\circ f_t^{-1}).$ Lemma \ref{conv2} is thus a restatement of Lemma \ref{conv}.
\end{proof}
\end{subsection}

\begin{subsection}{Proof of Theorem A}
Setting out notation for the main proof, let $\phi_1,\dots, \phi_n\in \textrm{QD}(S)$ and let $\rho$ be the action on the product of $\R$-trees $X$ associated to the $\phi_i$'s, with harmonic map $\pi=(\pi_1,\dots, \pi_n)$ and energy functional $\mathbf{E}_\rho.$ 

If any other equivariant harmonic map $\tau=(\tau_1,\dots, \tau_n):S'\to X$ were related to $\pi$ by quasiconformal maps, in the sense that one could find quasiconformal $f_1,\dots, f_n:S\to S'$ such that $\tau_i = \pi_i\circ f_i^{-1},$ then the Reich-Strebel formula (\ref{RSorig}) implies Theorem A. But it is not possible to factor harmonic maps to $\R$-trees via quasiconformal maps: in Figure 1 below, all foliations project to the same tree, but there is no leaf-preserving homeomorphism between the two bottom spaces.

Instead, we can approximate leaf-space projections $\pi_i$ and $\tau_i$ by harmonic diffeomorphisms between surfaces, $h_i^t$ and $g_i^t$, and then intertwine the diffeomorphisms by quasiconformal maps $f_i^t$. By the behaviour of high energy harmonic maps, once $t$ is sufficiently large, the $f_i^t$'s nearly intertwine the foliations associated to the $\pi_i's$ and $\tau_i$'s, and consequently each $\tau_i$ should differ from $\pi_i\circ (f_i^t)^{-1}$ by a perturbation. This last statement is the geometric intuition behind Lemma \ref{conv2}, which serves as our substitute for factoring harmonic maps.
 \begin{figure}[ht]
     \centering
     \includegraphics[scale=0.4]{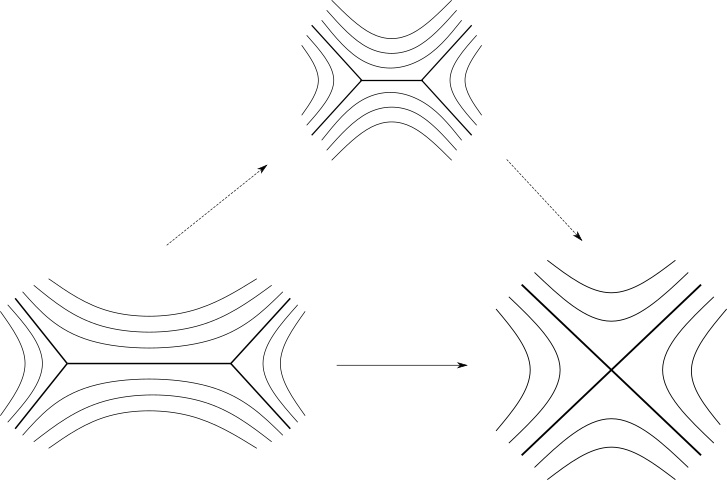}
     \caption{The map to the foliated space up top could be represented by $f_i^t,$ which could limit to the map to the right as $t\to\infty$, eventually collapsing the leaf connecting the singular points.}
 \end{figure}  
Our final preparation is a simple lemma.
\begin{lem}\label{easylemma}
    Let $\phi\in \QD(S)$ and let $\sigma_t$ be the hyperbolic metric such that the identity map $h_t:S\to (M,\sigma_t)$ has Hopf differential $t\phi.$ Then $$\frac{e(h_t)}{t}\geq 2|\phi|\nu^{-1}.$$
\end{lem}
\begin{proof}
    Applying Cauchy-Schwarz to the expression (\ref{mhopf}) yields $e(h_t)\geq 2t|\phi|\nu^{-1}.$ We then divide by $t.$
\end{proof}

\begin{proof}[Proof of Theorem A]
The forward direction is explained in Section 3.1, so we don't address it here. For the other direction, assume the new main inequality holds for $\phi_1,\dots, \phi_n$ and any choice of mutually homotopic quasiconformal maps $f_i:S\to S'$.  For each $t>0$, let $\sigma_1^t,\dots,\sigma_n^t$ be the hyperbolic metrics such that the identity maps $h_i^t:S\to (M,\sigma_i^t)$ are harmonic and have Hopf differentials $t\phi_i/4$. Let $\rho_t$ be the product of the holonomy representations for the $\sigma_i^t$'s, and let $\mathbf{E}_{\rho_t}$ be the associated energy functional.

Let $S'$ be another Riemann surface structure on $\Sigma_g$, let $g_i^t:S'\to (M,\sigma_t)$ be the harmonic map in the homotopy class of the identity, and set $f_i^t=(g_i^t)^{-1}\circ h_i^t.$ Applying Lemma \ref{conv2} $n$ times, we have that 
\begin{equation}\label{ntimes}
    \lim_{t\to\infty}\sum_{i=1}^n\frac{\mathcal{E}(S',h_i^t\circ (f_i^t)^{-1})}{t}=\mathbf{E}_\rho(S').
\end{equation}
Combining (\ref{ntimes}) with Lemma \ref{conv} applied to $S'=S$, and using the Reich-Strebel formula (\ref{RSorig}), we find
\begin{align*}
    \mathbf{E}_\rho(S')-\mathbf{E}_\rho(S)&=\lim_{t\to\infty}\frac{1}{t}\Big (\sum_{i=1}^n\mathcal{E}(S',\tilde{h}_i^t\circ (\tilde{f}_i^t)^{-1})- \mathbf{E}_{\rho_t}(S)\Big ) \\
    &=\lim_{t\to\infty} \frac{1}{t}\Big (\frac{1}{2}\sum_{i=1}^n \int_S e(h_i^t)\cdot \frac{|\mu_i^t|^2}{1-|\mu_i^t|^2}dA- \sum_{i=1}^n \textrm{Re} \int_S t\phi_i\cdot \frac{ \mu_i^t}{1-|\mu_i^t|^2}\Big ) \\
    &= \lim_{t\to\infty}\frac{1}{2}\sum_{i=1}^n \int_S \frac{e(h_i^t)}{t}\cdot \frac{|\mu_i^t|^2}{1-|\mu_i^t|^2}dA- \sum_{i=1}^n \textrm{Re} \int_S \phi_i\cdot \frac{ \mu_i^t}{1-|\mu_i^t|^2}
\end{align*}
Set $F(t)$ to be the quantity that we're taking a limit on, $$F(t)=\frac{1}{2}\sum_{i=1}^n \int_S \frac{e(h_i^t)}{t}\cdot \frac{|\mu_i^t|^2}{1-|\mu_i^t|^2}dA- \sum_{i=1}^n \textrm{Re} \int_S \phi_i\cdot \frac{ \mu_i^t}{1-|\mu_i^t|^2}.$$
By Lemma \ref{easylemma}, we know that for every $i,$ $\frac{1}{2}\cdot\frac{e(h_i^t)}{t}>|\phi_i|\nu^{-1}$, and hence for every $t,$
$$F(t)\geq \sum_{i=1}^n \int_S |\phi_i|\cdot \frac{|\mu_i^t|^2}{1-|\mu_i^t|^2}dA- \sum_{i=1}^n \textrm{Re} \int_S \phi_i\cdot \frac{ \mu_i^t}{1-|\mu_i^t|^2}.$$ Applying the new main inequality, we deduce that $F(t)\geq 0$ for every $t,$ and consequently that $\mathbf{E}_\rho(S)\leq \mathbf{E}_\rho(S'),$ which is the desired result.
 \end{proof}

With the proof of Theorem A complete, let's give a comment on why the new main inequality is special to the leaf space projections. Any equivariant harmonic map to an $\R$-tree is the composition of a leaf space projection and a map that folds segments onto each other (see \cite{FW} and \cite[Section 4.1]{MSS}). Two harmonic maps to the same $\R$-tree can arise from foldings of different leaf spaces. Consequently, the critical leaves for the Hopf differentials can look quite different, and we can't expect to be able to find quasiconformal maps that nearly intertwine the critical leaves.

In a general setting, it should be more promising to study maps to $\R$-trees that are nearby. One could perturb a variation of maps so that the critical structure is fixed, which eliminates the issue raised above. The most efficient way to perturb is to use the log cut-off trick, which negligibly affects the second variation of energy, but can force the third variation to blow up. Hence, for other maps to $\R$-trees, such as the maps to $\R^n$ in the next section, the best one can hope for is the infinitesimal version of the new main inequality.
\end{subsection}

\end{section}

\begin{section}{Classical minimal surfaces}\label{r3}
We return to the setup from Section 1.2: $h=(h_1,\dots, h_n):\overline{\mathbb{D}}\to\R^n$ is a non-constant admissible minimal map with Weierstrass-Enneper data $\alpha=(\alpha_1,\dots,\alpha_n)$. We denote the Hopf differential of $h_i$ by $\phi_i=\alpha_i^2.$

We first prove Theorem C, which is then used to prove Theorem B. We conclude with Theorem D. 
\begin{subsection}{Variations by quasiconformal maps}
To properly begin, we need to explain how to vary quasiconformal maps.
\begin{defn}
Beltrami forms $\mu,\nu\in L_1^\infty(\mathbb{D})$ are equivalent if the normal solutions $f^\mu$ and $f^\nu$ agree on $\mathbb{C}\backslash\mathbb{D}.$
\end{defn}
The universal Teichm{\"u}ller space $\mathbf{T}$ has many definitions, and the resulting spaces can all be identified in a reasonable way. The model we take is $\mathbf{T}=L_1^\infty(\mathbb{D})/\sim,$ where $\mu\sim \nu$ if $\mu$ and $\nu$ are equivalent.
\begin{remark}
It is more common to define $\mathbf{T}$ by taking $F^\mu=f^\mu/f^\mu(1)$ instead of $f^\mu$. Under our definition, tangent vectors at $[\mu]=[0]$ have a more tractable expression.
\end{remark}
Tangent vectors in $T_{[0]}\mathbf{T}$ should arise from functions in $L^\infty(\mathbb{D})$ up to a certain identification. To make this identification explicit, we first recall the operator $P$, defined on $L^p(\mathbb{C})$, $2<p<\infty,$ by $$P(h)(z) = -\frac{1}{\pi}\int_{\mathbb{C}}h(\zeta)\Big (\frac{1}{\zeta-z} -\frac{1}{\zeta}\Big) dxdy.$$
Secondly, the Beurling transform $T$ is defined on $C_0^\infty(\mathbb{C})$ by the principal value $$T(h)(z) =\lim_{\epsilon\to 0} -\frac{1}{\pi}\int_{|\zeta-z|>\epsilon}\frac{h(\zeta)}{(\zeta-z)^2}dxdy,$$ and extends continuously to $L^p(\C)$, $1<p< \infty.$

For $h\in L^\infty(\mathbb{D}),$ we extend to $\mathbb{C}$ by setting $h=0$ on $\mathbb{C}\backslash\mathbb{D},$ and we write $P(h)$ and $T(h)$ for $P$ and $T$ applied to the extension of $h$.
The normal solution to the Beltrami equation for $\mu\in L_1^\infty(\mathbb{D})$ can be written explicitly in terms of $P$ and $T$:
 $$f^\mu(z) = z+P(\mu)(z)+P(\mu T(\mu))(z) + P(\mu T(\mu T(\mu)))(z)+\dots$$
 So, if $\mu=t\dot{\mu}+o(t)$ is a variation of Beltrami forms, then the normal solution along the variation is $$f^{\mu_t}= z + tP(\dot{\mu})+o(t).$$ Therefore, $\dot{\mu},\dot{\nu}\in L^\infty(\mathbb{D})$ give the same variation in $\mathbf{T}$ if and only if $P(\dot{\mu})=P(\dot{\nu})$ on $\mathbb{C}\backslash\mathbb{D}.$
 \begin{defn}
     $\dot{\mu},\dot{\nu}\in L^\infty(\mathbb{D})$ are infinitesimally equivalent if $P(\dot{\mu})=P(\dot{\nu})$ on $\mathbb{C}\backslash\mathbb{D}.$
 \end{defn}
\begin{defn}
    The space $\mathcal{V}$ from the introduction, our model for $T_{[0]}\mathbf{T}$, is obtained by restricting every function of the form $P(h)$, $h\in L^\infty(\mathbb{D}),$ to $\mathbb{C}\backslash\mathbb{D}$.
\end{defn}
 In order to show that we can pick variations with lots of freedom, which we'll do to prove Theorems C and D, we justify the well known fact below.
 \begin{prop}\label{choose}
     For every $f\in C_0^\infty(\mathbb{C})$ that is holomorphic on $\C\backslash\mathbb{D}$, we can find $\dot{\mu}\in L^\infty(\mathbb{D})$ with $P(\dot{\mu})=f.$
 \end{prop}
The following basic result can be verified immediately.
\begin{prop}\label{Cauchyprop}
Assume $h\in C_0^\infty(\mathbb{C})$. Then $P(h)$ is smooth, $(P(h))_{\overline{z}}=h$, and $P(h)(z)$ tends to $0$ as $|z|\to \infty$. 
\end{prop}
\begin{proof}[Proof of Proposition \ref{choose}]
    Let $f\in C_0^\infty(\mathbb{C})$ be holomorphic in $\C\backslash\mathbb{D}$. Define the function $\dot{\mu}$ on $\mathbb{C}$ by $\dot{\mu}=f_{\overline{z}}.$ By Proposition \ref{Cauchyprop}, $(P(\dot{\mu}))_{\overline{z}}=f_{\overline{z}}$, so $(f-P(\dot{\mu}))$ is an entire function that is bounded, and therefore a constant. Since both $f(z)$ and $P(\dot{\mu})(z)$ tend to $0$ as $|z|\to \infty,$ they are identically equal. Hence, this $\dot{\mu}$ satisfies $P(\dot{\mu}) = f$.
\end{proof}
 Now we can formulate our problem more precisely. Recall from Section 2.2 that for harmonic functions to $\R$, the Reich-Strebel computation gives the following.
\begin{lem}\label{RSr3}
Let $h: \mathbb{D}\to\mathbb{R}$ be a harmonic function with integrable Hopf differential $\phi$, and $f:\mathbb{C}\to\mathbb{C}$ a quasiconformal map with Beltrami form $\mu$. The formula
\begin{equation}\label{RSforr3}
    \mathcal{E}(h\circ f^{-1}) -\mathcal{E}(h) =  -4\textrm{Re} \int_{\mathbb{D}} \phi\cdot \frac{ \mu}{1-|\mu|^2}dxdy + 4\int_{\mathbb{D}} |\phi|\cdot \frac{|\mu|^2}{1-|\mu|^2}dxdy.
\end{equation}
holds.
\end{lem}
 We call paths $\mu_i(t): [0,t_0]\to L_1^\infty(\mathbb{D})$ equivalent if they project to the same path in $\mathbf{T}$. We fix any $\varphi\in \mathcal{V}$ and look for mutually equivalent $C^2$ paths $\mu_i^t$ tangent at time zero to $\varphi$ in $\mathbf{T}$, such that if $f_i^t$ is the normal solution at time $t$, then $$\frac{d^2}{dt^2}|_{t=0}\sum_{i=1}^n\mathcal{E}(f_i^t(\Omega),h_i\circ (f_i^t)^{-1})<0.$$ As we noted in the introduction, since energy dominates area, it follows that the variation $h_t=(h_1\circ (f_1^t)^{-1},\dots, h_1\circ (f_n^t)^{-1})$ decreases the area to second order.
\end{subsection}

\begin{subsection}{The second variation of energy} 
 In \cite[Lemma 3.2]{M2} and \cite[Proposition 4.2]{M2}, the author computes the second variation of the new main inequality. In our context, this is the second variation of the energy. We recap the computation here. 
 \begin{prop}\label{2varT} If $\dot{\mu}_1,\dots, \dot{\mu}_n\in L^\infty(\mathbb{D})$ are mutually infinitesimally equivalent, then there exists $C^2$ mutually equivalent paths $\mu_i(t): [0,t_0]\to L_1^\infty(\mathbb{D})$ tangent to $\dot{\mu}_i$ at $t=0$ and with normal solutions $f_i^t$ such that
 \begin{equation}\label{2var}
 \frac{d^2}{dt^2}|_{t=0}\sum_{i=1}^n \mathcal{E}(h_i\circ (f_i^t)^{-1})=4\textrm{Re}\sum_{i=1}^n\int_{\mathbb{D}}\phi_i\dot{\mu}_iT(\dot{\mu}_i)dxdy+4\sum_{i=1}^n\int_{\mathbb{D}}|\phi_i||\dot{\mu_i}|^2dxdy.
 \end{equation}
 \end{prop}
 \begin{proof}
 Let $\mu_i(t) = t\dot{\mu}_i + t^2\ddot{\mu}_i +o(t^2)$ be mutually equivalent paths with normal solutions $f_i^t$. Differentiating the Reich-Strebel formula (\ref{RSforr3}),
\begin{align*}
    \frac{1}{4}\frac{d^2}{dt^2}|_{t=0}\sum_{i=1}^n \mathcal{E}(h_i\circ (f_i^t)^{-1}) = -\textrm{Re}\sum_{i=1}^n\int_{\mathbb{D}}\phi_i\ddot{\mu_i}dxdy+\sum_{i=1}^n|\phi_i||\dot{\mu_i}|^2dxdy
\end{align*}
(see \cite[Lemma 3.2]{M2} for details). Crucially making use of the fact that $\sum_{i=1}^n\phi_i=0$, i.e., that $h$ is a minimal map, it follows from \cite[Proposition 4.2]{M2} that we can choose mutually equivalent paths such that $$\textrm{Re}\sum_{i=1}^n\int_{\mathbb{D}}\phi_i\ddot{\mu_i}dxdy=-\textrm{Re}\sum_{i=1}^n\int_{\mathbb{D}}\phi_i\dot{\mu}_iT(\dot{\mu}_i)dxdy.$$
Putting the pieces together gives the result.
 \end{proof}
\begin{remark}
Up to this point, we have not used that $\phi_i=\alpha_i^2$. So in particular, Proposition (\ref{2varT}) holds as well for minimal maps to $\R$-trees. 
\end{remark}
It is computed in \cite[Section 6]{M2}, using the relation $(P(h))_z=Th$ (distributionally), that 
\begin{equation}\label{P1}
    -\textrm{Re}\sum_{i=1}^n\int_{\mathbb{D}}\phi_i\dot{\mu}_iT(\dot{\mu}_i)dxdy=\textrm{Re}\sum_{i=1}^n\int_{\mathbb{D}}(\alpha_i P(\dot{\mu}_i))_{z}(\alpha_i P(\dot{\mu}_i))_{\overline{z}}dxdy,
\end{equation}
 and 
 \begin{equation}\label{P2}
     \sum_{i=1}^n \int_{\mathbb{D}} |\phi_i||\dot{\mu}_i|^2dxdy=\sum_{i=1}^n \int_{\mathbb{D}} |(\alpha_i P(\dot{\mu}_i))_{\overline{z}}|^2dxdy.
 \end{equation}
Substituting (\ref{P1}) and (\ref{P2}) into (\ref{2var}), we arrive at the following
\begin{prop}\label{Cauchy} If $\dot{\mu}_1,\dots, \dot{\mu}_n\in L^\infty(\mathbb{D})$ are mutually infinitesimally equivalent, then there exists $C^2$ mutually equivalent paths $\mu_i(t): [0,t_0]\to L_1^\infty(\mathbb{D})$ tangent to $\dot{\mu}_i$ at $t=0$ and with normal solutions $f_i^t$ such that
\begin{align*}
    \frac{d^2}{dt^2}|_{t=0}\sum_{i=1}^n \mathcal{E}(h_i\circ (f_i^t)^{-1}) &= 4\textrm{Re}\sum_{i=1}^n\int_{\mathbb{D}}(\alpha_i P(\dot{\mu}_i))_{z}(\alpha_i P(\dot{\mu}_i))_{\overline{z}}dxdy +4\sum_{i=1}^n\int_{\mathbb{D}} |(\alpha_i P(\dot{\mu}_i))_{\overline{z}}|^2dxdy \\
    &=4\sum_{i=1}^n \mathcal{F}(\alpha_iP(\dot{\mu}_i)),
\end{align*}
where $\mathcal{F}$ is the function from Section 1.2.
\end{prop}
\end{subsection}
\begin{subsection}{Proof of Theorem C}
We continue in the setting above with an admissible $h$ with Weierstrass-Enneper data $\alpha=(\alpha_1,\dots,\alpha_n),$ and $\phi_i=\alpha_i^2$ We fix a variation $\varphi\in\mathcal{V}.$

Proposition \ref{Cauchy} says that if we are given $\varphi\in \mathcal{V}$ and we can find maps $P(\dot{\mu}_1),\dots, P(\dot{\mu}_n)$ on $\mathbb{D}$ extending to $\varphi$ on $\C\backslash\mathbb{D}$ such that $\sum_{i=1}^n \mathcal{F}(\alpha_i P(\dot{\mu}_i))<0,$ then $\varphi$ destabilizes $h$.
The first question is, how to pick $P(\dot{\mu}_i)$ that have the best chance of destabilizing $h$? If we could pick $P(\dot{\mu}_i)$ so that there is a choice of quasiconformal maps $f_i^t(z)=z+tP(\dot{\mu}_i)(z)+o(t)$ such that $h_i\circ (f_i^t)^{-1}$ is harmonic, then $h_i\circ (f_i^t)^{-1}$ would minimize the energy over maps with the same boundary values at each time $t$. Recalling the local pictures from Section 3, picking such $f_i^t$ is not in general possible. 

However, we can still argue heuristically. Given some choice of $P(\dot{\mu}_i)$ and accompanying variation of quasiconformal maps $f_i^t,$ define $\dot{h}_i:\overline{\mathbb{D}}\to\R$ by $$h_i\circ (f_i^t)^{-1}=h_i+t\dot{h}_i+o(t).$$ Since the Laplacian is linear, if we demand that $\dot{h}_i$ allows a variation of harmonic functions, then $\dot{h}_i$ must be a harmonic function itself. Up to first order, the inverse of $f_i^t$ is $$(f_i^t)^{-1}(z)=z-tP(\dot{\mu}_i)(z) + o(t).$$ Computing via the chain rule, $$\dot{h}_i = \frac{d}{dt}|_{t=0}h_i\circ (f_i^t)^{-1}= -2\textrm{Re}(\alpha_i P(\dot{\mu}_i)).$$ Let $v_i$ be the harmonic extension of the complex-valued function $(\frac{\partial}{\partial z} h)\cdot \varphi|_{\partial\mathbb{D}}$.  If we pretend that we can pick $P(\dot{\mu}_i)$ to be $(\frac{\partial}{\partial z} h)^{-1}v_i,$ then the choice would minimize the map $$(g_1,\dots, g_n)\mapsto\sum_{i=1}^n\mathcal{F}(\alpha_i g_i),$$ where the $g_i$ range over every map extending $\varphi$, since the corresponding path $f_i^t$ would minimize the second derivative of $\mathcal{E}(h_i\circ (f_i^t)^{-1})$ at time zero. The problem of course is that these choices for $P(\dot{\mu}_i)$ blow up at the zeros of $(\frac{\partial}{\partial z} h_i)$. We're saved by the log cut-off trick, which allows us to smoothly perturb $v_i$ to be zero in a neighbourhood of the zero set of $(\frac{\partial}{\partial z} h_i)$, so that the division is possible, while only changing the evaluation of $\mathcal{F}$ by a controlled amount. The computation for the functional $\mathcal{F}$ is carried out in \cite[Section 5]{M2}.
\begin{prop}[Proposition 5.1 in \cite{M2}]\label{mar}
Let $Z\subset \mathbb{D}$ be a finite set of points and $f:\overline{\mathbb{D}}\to\mathbb{C}$ a smooth function. Then for every $\epsilon>0$, there exists smooth $g:\overline{\mathbb{D}}\to\mathbb{C}$ such that 
\begin{enumerate}
    \item $f(z)=g(z)$ for $z$ in a neighourhood of $\partial\mathbb{D}.$
    \item $g(z)=0$ for $z$ in some neighbourhood of each $z_0\in Z$.
    \item $|\mathcal{F}(f)-\mathcal{F}(g)|<\epsilon.$
\end{enumerate}
\end{prop}
We're ready for the formal proof of the theorem.
\begin{proof}[Proof of Theorem C]
Suppose $$\mathcal{F}_{\alpha}(\varphi) := \sum_{i=1}^n \mathcal{F}(v_i)<0.$$ Let $\epsilon>0$ be small enough so that 
\begin{equation}\label{epsilon}
    \mathcal{F}_{\alpha}(\varphi)+\epsilon<0.
\end{equation}
Let $Z_i$ be the zero set of $\frac{\partial}{\partial z} h_i$, and apply Proposition \ref{mar} to $(v_i,Z_i)$ 
to find $g_i:\overline{\mathbb{D}}\to\mathbb{C}$ such that $g_i=(\frac{\partial}{\partial z} h_i)\cdot \varphi$ on $\partial\mathbb{D}$, and  
\begin{equation}\label{epsilonn}
    |\mathcal{F}(v_i)-\mathcal{F}(g_i)|<\frac{\epsilon}{n}.
\end{equation}
Via Proposition \ref{choose}, we can choose $\dot{\mu_i}$ so that $P(\dot{\mu}_i)=\alpha_i^{-1}g_i$. By (\ref{epsilon}) and (\ref{epsilonn}), $$\sum_{i=1}^n \mathcal{F}(\alpha_ig_i)<0.$$ Theorem C now follows from Proposition \ref{Cauchy}.
\end{proof}

    Theorem C can probably also be proved by using the destabilizing strategy mentioned in the introduction of varying the boundary parametrization and taking harmonic extensions. To understand how to relate the two methods, we need to know how to turn $\varphi$ into a variation of boundary parametrizations. $\mathbf{T}$ is also the space of quasisymmetric maps of $\partial\mathbb{D}$ mod M{\"o}bius transformations. In this model, the tangent space at the identity identifies with the Zygmund class of vector fields on $\partial\mathbb{D}$ \cite[Section 2]{Nag}. Nag finds a beautiful identification of the tangent spaces to the different models in \cite[Section 3]{Nag}, which explains how to get a Zygmund vector field out of an admissible holomorphic map on $\mathbb{C}\backslash\mathbb{D}.$ We gave the proof of Theorem C because it is interesting to see it from our angle, and because elements of the proof will be used toward Theorem B.
\end{subsection}

\begin{subsection}{The self-maps index}\label{selfmapss}
Continuing in our usual setting and keeping the notation from above, we now prove Theorem B and its corollary.
\begin{defn}
    The real quadratic form $\mathbf{L}_h : \mathcal{V} \to \R$ is defined by $\mathbf{L}_h(\varphi)= \sum_{i=1}^n \mathcal{F}(v_i),$ where $v_i$ is the harmonic extension of $(\frac{\partial}{\partial z} h)\cdot \varphi|_{\partial\mathbb{D}}.$ The self-maps index is the maximum dimension of a subspace on which $\mathbf{L}_h$ is negative definite.
\end{defn}
Noting that taking the Poisson extension is a linear operation, it is routine to check that $\mathbf{L}_h$ is a real quadratic form.
 
Let $m$ be the Euclidean metric on $\R^n$, and denote the volume form by $dV$. The area of a $C^2$ map $g$ from a domain $\Omega\subset \C$ to $\R^n$ is the area of the image $g(\Omega)\subset \R^n$, $$A(\Omega,g):=\int_{\Omega} g^*dV.$$ $h$ may be only a branched immersion, but it is well-understood that the normal bundle, apriori defined where $h$ is regular, extends real analytically over the branch points (see, for example, \cite[Lemma 1.3]{GOR}). This extension of the normal bundle is denoted $N_h\subset h^* T\R^n$. Variations of the image surface are elements of $\Gamma_0(N_h),$ the space of $C^\infty$ sections of $N_h$ that extend to zero on $\partial\mathbb{D}$, which we tacitly view as functions $X:\mathbb{D}\to\R^n.$ The second variation of area is defined by a real quadratic form  $\mathbf{Q}_h:\Gamma_0(N_h)\to \R$, $$\mathbf{Q}_h(X) = \frac{d}{dt}|_{t=0} A(\Omega, h+tX)$$ (see \cite[Theorem 32]{Law} for the well known formula for the right hand side).
The usual index $\textrm{Ind}(h)$ is the maximal dimension of a subspace on which $\mathbf{Q}_h$ is negative definite. Theorem D is the statement that $\textrm{Ind}(\mathbf{L}_h)=\textrm{Ind}(h).$ Before we enter the proof, we recall the following application of the log cut-off trick in its usual form (see [Section 4.4, MSS] for detailed explanation).

\begin{prop}\label{cutoffindex}
    Let $\textrm{Ind}_0(h)$ be the index of $h$ restricted to variations in $\Gamma_0(N_h)$ that vanish on a neighbourhood of the critical points of every $h_i$. Then $\textrm{Ind}(h)=\textrm{Ind}_0(h).$
\end{prop}

\begin{proof}[Proof of Theorem B]
It was already explained in Section 4.1 that a destabilizing self-maps variation yields a variation of maps $h_t:\overline{\mathbb{D}}\to \R^n$ that decreases area to second order. Pulling back the Euclidean metric from $T\R^n$ to $h^*T\R^n$ and orthogonally projecting the induced section of $h^*T\R^n$ onto $N_h$, we obtain a section $X\in \Gamma_0(N_h)$ with $\mathbf{Q}_h(X)<0.$ 

To prove the theorem, we need to show that if $X\in \Gamma_0(N_h)$ vanishes in a neighbourhood of the critical point of every $h_i$ and destabilizes the area of $h$, then we can find a destabilizing self-maps variation in a way that inverts the process above. For then $\textrm{Ind}(\mathbf{L}_h)=\textrm{Ind}(\mathbf{Q}_h),$ and we can appeal to Proposition \ref{cutoffindex}.

We will apply Theorem C by finding a variation $\varphi\in\mathcal{V}$ with $\mathcal{F}_\alpha(\varphi)<0$. Set $h_t=h+tX.$ If $h$ has branch points, then the pullback metric $h^*m$ is degenerate at those points, and regular elsewhere. $h^*m$ is conformal to the flat metric $\sigma(z)=|dz|^2$ on $\mathbb{D}$ in the sense that there is a bounded and $C^\infty$ function $u:\mathbb{D}\to [0,\infty)$ with isolated zeros exactly at the branch points of $h$, and such that $h^*m= u\sigma.$ Since $X=0$ in $U,$ $h_t^*m=h^*m$ in $U$.

There exists $t_0>0$ such that for $t<t_0$, the degenerate locus of $h_t^*m$ is equal to that of $h^*m$. We define a family of non-degenerate $C^\infty$ metrics $(\sigma_t)_{t<t_0}$ on $\mathbb{D}$ by
$$\sigma_t(z)=
\begin{cases}
\sigma(z), \hspace{1mm} z\in U\\
u(z)^{-1}h_t^*m(z), \hspace{1mm} z \in\mathbb{D}\backslash U\\
\end{cases}.$$
We emphasize that $h_t^*m$ is not necessarily conformally flat. For each $t\leq t_0$, by the measurable Riemann mapping theorem, Theorem \ref{RMT}, we can find a Jordan domain $\Omega_t\subset \mathbb{C}$ and a quasiconformal homeomorphism $f_t:\mathbb{D}\to\Omega_t$ that takes $\sigma_t$ to a conformally flat metric (this is a classical application). Observe that the Beltrami form $\mu_t$ of each $f_t$ extends to $0$ on $\partial\mathbb{D}$, since $X$ extends to $0$ on $\partial\mathbb{D}.$ For each $t$, we extend $\mu_t$ to $0$ on $\mathbb{C}\backslash\mathbb{D}.$ We then take the $L^\infty$ function $\dot{\mu}=\frac{d}{dt}|_{t=0}\mu_t$ and the associated tangent vector $\varphi=P(\dot{\mu})|_{\mathbb{C}\backslash\mathbb{D}}\in\mathcal{V}.$ This is the desired self-maps variation.

Let's now verify Theorem C for $\varphi$. Note that for every $t$, the map $h\circ f_t^{-1}:\Omega_t\to \R^n$ is weakly conformal. Obviously, the area of $h\circ f_t^{-1}(\Omega_t)$ is equal to area of $h_t(\mathbb{D}).$ By design, the maps $h\circ f_t^{-1}$ are weakly conformal, and therefore
$$A(\Omega_t,h\circ f_t^{-1})=\mathcal{E}(\Omega_t,h\circ f_t^{-1}).$$ Replacing each $h_i\circ f_t^{-1}$ with the harmonic extension of the boundary map, say $v_i^t,$ cannot increase the energy. Hence, $$\mathcal{E}(\Omega_t,v_i^t)\leq \mathcal{E}(\Omega_t, h\circ f_t^{-1})=A(\Omega_t,h\circ f_t^{-1})=A(\Omega,h_t).$$ Taking the second derivative at time zero, we obtain $$\mathcal{F}_\varphi(\alpha)\leq \mathbf{Q}_h(X)<0.$$ As discussed, by Theorem C we are done.
\end{proof}

\begin{proof}[Proof of Corollary B]
By Theorem B, $h$ is stable if and only if $\textrm{Ind}(\mathbf{Q}_h)=0.$ By Proposition \ref{2varT}, $\textrm{Ind}(\mathbf{Q}_h)=0$ if and only if the infinitesimal new main inequality holds for the Hopf differentials of the component maps and all choices of infinitesimally equivalent $\dot{\mu}_1,\dots, \dot{\mu}_n$.
\end{proof}

\end{subsection}

\begin{subsection}{Explicit destabilizing variations}
To conclude the paper, we test out the framework we've developed and prove Theorem D. We compute the functional $\mathcal{F}_{\alpha}(\varphi)$ for polynomial Weierstrass data $\alpha=(\alpha_1,\dots, \alpha_n)$ and the variation $\varphi(z)=\gamma z^{-m}.$ Recall from the introduction that we have defined, for a polynomial $p(z) = \sum_{j=0}^r a_jz^j,$ $\gamma\in\mathbb{C}^*$, and $m>0,$  \begin{equation}\label{Cpsi}
    C(p,\gamma,m) =\pi\sum_{j=0}^{m-1}\frac{\textrm{Re}(\gamma^2a_ja_{2m-j})+|\gamma|^2|a_j|^2}{m-j}.
\end{equation}
Setting $\alpha(z)=p(z)dz,$ the harmonic extension of $p\cdot \varphi|_{\partial\mathbb{D}}$ is $$f_{p,\gamma,m}(z)=\gamma(a_0\overline{z}^m+\dots + a_m + a_{m+1}z + \dots a_nz^{n-m}).$$ 
\begin{lem}\label{form}
In the setting above, $\mathcal{F}(f_{p,\gamma,m})=C(p,\gamma,m).$
\end{lem}
\begin{proof}
For notations sake, set $f=f_{p,\gamma,m}$. We compute the integrals individually. First, 
\begin{equation}\label{firstintegrand}
    |f_{\overline{z}}|^2= |\gamma|^2\sum_{j=0}^{m-1}|a_j|^2|z|^{2(m-1-j)}+2|\gamma|^2\textrm{Re}\sum_{j=0}^{m-1}\sum_{k\neq j}a_j\overline{a_k}\overline{z}^{m-1-j}z^{m-1-k}.
\end{equation}
Due to $L^2$-orthogonality of the Fourier basis on $S^1,$ the second term on the right in (\ref{firstintegrand}) vanishes upon integration: 
\begin{align*}
    &2|\gamma|^2\textrm{Re}\sum_{j=0}^{m-1}\sum_{k\neq j}a_j\overline{a_k}\int_{\mathbb{D}}\overline{z}^{m-1-j}z^{m-1-k}|dz|^2 \\
    &= 2|\gamma|^2\textrm{Re}\sum_{j=0}^{m-1}\sum_{k\neq j}a_j\overline{a_k}\int_0^1 r^{2m-1-j-k}dr \int_0^{2\pi} e^{i\theta(j-k)} d\theta =0.
\end{align*}
Hence,
\begin{equation}\label{int1}
    \int_{\mathbb{D}}|f_{\overline{z}}|^2 = 2\pi|\gamma|^2\sum_{j=0}^{m-1}|a_j|^2\int_0^1 r^{2m-1-2j} dr= \pi|\gamma|^2 \sum_{j=0}^{m-1}\frac{|a_j|^2}{m-j}.
\end{equation}
The term $f_zf_{\overline{z}}$ is a sum of terms of the form $c_{j,k}\overline{z}^{m-j}z^{r-m-k}$. Again by $L^2$-orthogonality, the integration over the disk evaluates to a non-zero number if and only if $0\leq j \leq m-1$, $m+1\leq k \leq r$, and
$(m-1)-j = (r-(m+1))-(r-k)$, i.e., $k=2m-j$. This returns the formula
\begin{equation}\label{int2}
    \textrm{Re}\gamma^2\int_{\mathbb{D}} f_zf_{\overline{z}}=\textrm{Re}\gamma^2\sum_{j=0}^{m-1} a_j a_{2m-j}\int_\mathbb{D}|z|^{2(m-1-j)}|dz|^2= \pi\textrm{Re}\gamma^2\sum_{j=0}^{m-1} \frac{a_j a_{2m-j}}{m-j}.
\end{equation}
 Putting (\ref{int1}) and (\ref{int2}) together, $$\mathcal{F}(f) = \pi\sum_{j=0}^{m-1}\frac{\textrm{Re}(\gamma^2a_ja_{2m-j})+|\gamma|^2|a_j|^2}{m-j}.$$
\end{proof}
\begin{proof}[Proof of Theorem D]
Apply Theorem C with the variation $\gamma z^{-m}$, using Lemma \ref{form} $n$ times to obtain the value of $\mathcal{F}_\alpha(\varphi)$ . 
\end{proof}

\end{subsection}

\bibliographystyle{plain}
\bibliography{bibliography}

\end{section}

\end{document}